\documentclass[12pt]{article}

\usepackage{amsmath,amsthm,amssymb,dsfont,enumerate,tikz,pgfplots,hyperref}
\usepackage{setspace}
\usepackage[margin=1.25in]{geometry}

\newtheorem{theorem}{Theorem}[section]
\newtheorem{proposition}[theorem]{Proposition}

\theoremstyle{definition}
\newtheorem{definition}[theorem]{Definition}

\newtheorem{remark}[theorem]{Remark}

\numberwithin{equation}{section}
\numberwithin{figure}{section}

\renewcommand{\subset}{\subseteq}
\renewcommand{\hat}{\widehat}

\renewcommand{\epsilon}{\varepsilon}

\def\supp{\text{supp}}
\def\vol{\text{vol}}

\def\<{\langle}
\def\>{\rangle}
\def\({\Big(}
\def\){\Big)}

\def\F{\mathcal{F}}
\def\G{\mathcal{G}}

\def\P{\mathcal{P}}
\def\Q{\mathcal{Q}}
\def\R{\mathbb{R}}
\def\r{\mathcal{R}}
\def\S{\mathcal{S}}
\def\s{\mathbb{S}}

\def\Z{\mathbb{Z}}

\title{Rotationally invariant time-frequency scattering transforms}

\author{Wojciech Czaja\thanks{Department of Mathematics, University of Maryland, College Park} \and Weilin Li\footnotemark[1]}

\date{}

\begin{document}

\maketitle

\begin{abstract}
In this paper we construct directionally sensitive functions that can be viewed as directional time-frequency representations. We call such a sequence a \emph{rotational uniform covering frame} and by studying rotations of the frame, we derive the \emph{rotational Fourier scattering transform} and the \emph{truncated rotational Fourier scattering transform}. We prove that both operators are rotationally invariant, are bounded above and below, are non-expansive, and contract small translations and additive diffeomorphisms. We also construct \emph{finite uniform covering frames}.
\end{abstract}


{\bf 2010 Math subject classification:} 42C15, 47N99, 68T10

\section{Introduction}

The extraction of informative and meaningful features from abstract data is an important step in many signal processing and data analysis applications. For example, a natural image typically comprises of a collection of objects, each having various shapes and patterns. Decomposing an image by orientation, location, and oscillation can separate and isolate individual objects from the remaining ones. For classification tasks, features also need to be stable to small distortions of the input data and encode underlying symmetries and invariances. 

Motivated by classification problems, Mallat constructed a stable feature extractor called the \emph{windowed scattering transform}, by incorporating a cascade of wavelet transforms and complex modulus \cite{mallat2012group}. Together with collaborators, they showed that it is useful for representing many types of data \cite{bruna2013invariant, sifre2013rotation, anden2014deep, hirn2015quantum}. Inspired by this theory and the use of Gabor representations in image analysis \cite{hamamoto1998gabor, kong2003palmprint, arivazhagan2006texture}, we developed a complementary time-frequency theory and studied the properties of the \emph{Fourier scattering transform} \cite{czaja2016analysis}.

The main purpose of this paper is to further incorporate directional information into the Fourier scattering transform. For many classification tasks, the data's information content is invariant under rotations. To name a few, previous papers have used rotationally invariant feature extractors for face recognition \cite{rowley1998rotation} and texture classification \cite{ojala2002multiresolution}. We approach this problem by building upon our previous work. We introduced a class of frames called \emph{uniform covering frames} and combined them with a neural network structure. We review these definitions in Section \ref{section background}. 

Uniform covering frames are not necessarily directionally sensitive, so in Section \ref{section rotational}, we construct a subclass called \emph{rotational uniform covering frames}. These functions are generated by modulations and a nested sequence of finite rotation groups. They are well-localized in both space and frequency, oscillate at various frequencies and directions, and their Fourier transforms are supported in wedges whose diameters are uniformly bounded. 

In Section \ref{section Fourier}, we exploit the rotational structure of these frame elements to derive the \emph{rotational Fourier scattering transform} and the \emph{truncated rotational Fourier scattering transform}. To do this, we study the action of a finite rotation group on the frame and this group theoretic viewpoint gives us a natural way of combining the frame with a neural network structure. Our main result, Theorem \ref{thm1}, shows that both scattering transforms are invariant under these rotations and satisfy other desirable feature extraction properties. 

The remaining sections answer natural questions about our theory. In Section \ref{section finite}, we show how to convert semi-discrete uniform covering frames to \emph{finite uniform covering frames}. These can be used to compute features for machine learning and signal processing tasks. In Section \ref{section directional}, we discuss the connection between rotational uniform covering frames and other types of directional representations, namely, curvelets, ridgelets, shearlets, and Gabor ridge functions. We argue that rotational uniform covering frames are time-frequency analogues of curvelets. 

\section{Background}
\label{section background}

Let $|\cdot|$ be the Euclidean distance on $\R^d$. The Fourier transform of a Schwartz function $f\in\S(\R^d)$ is $\hat f(\xi) =\int_{\R^d} f(x)e^{-2\pi ix\cdot \xi} dx$ and this definition has a unique extension to $f\in L^2(\R^d)$. Let $\<\cdot,\cdot\>$ be the usual inner product on $L^2(\R^d)$ and let $\|\cdot\|_{L^2}$ be the norm. For $y\in\R^d$, let $T_y$ be the translation operator $T_yf(x)=f(x-y)$. For $\tau\in C^1(\R^d;\R^d)$, let $T_\tau$ be the additive diffeomorphism $T_\tau f(x)=f(x-\tau(x))$. The essential support of a Lebesgue measurable function $f$, denoted $\supp(f)$, is the complement of the largest open set where $f=0$ almost everywhere. We say a function $f\in L^2(\R^d)$ is $(R,\epsilon)$ band-limited if $\|\hat f\|_{L^2(Q_R(0))}\geq (1-\epsilon)\|f\|_{L^2}$, where $Q_R(0)$ is the closed cube of side length $2R$ centered at 0.

\begin{definition}
	\label{def1}
	Let $\P$ be a countably infinite index set. A \emph{uniform covering frame} is a sequence of functions,
	\[
	\F=\{f_0\}\cup\{f_p\colon p\in\P\},
	\]
	satisfying the following assumptions. 
	\begin{enumerate}[(a)]\itemsep+.5em
		\item
		\underline{Assumptions on $f_0$ and $f_p$}. Let $f_0\in L^1(\R^d)\cap L^2(\R^d)\cap C^1(\R^d)$ such that $\hat{f_0}$ is supported in a compact neighborhood of the origin and $|\hat{f_0}(0)|=1$. For each $p\in\P$, let $f_p\in L^1(\R^d)\cap L^2(\R^d)$ such that $\supp(\hat{f_p})$ is compact and connected. 
		\item
		\underline{Frame condition}. Assume that for all $\xi\in\R^d$,
		\begin{equation}
			\label{eq frame} 
			|\hat{f_0}(\xi)|^2+\sum_{p\in\P} |\hat{f_p}(\xi)|^2=1.
		\end{equation}
		\item
		\underline{Uniform covering property}. For any $R>0$, there exists an integer $N>0$ such that for each $p\in\P$, the set $\supp(\hat{f_p})$ can be covered by $N$ cubes of side length $2R$.
	\end{enumerate}
\end{definition}

Uniform covering frames are generalizations of semi-discrete Gabor frames with band-limited windows and additionally, no wavelet frame is a uniform covering frame and vice versa \cite{czaja2016analysis}. We combined them with a network structure using the following method. Let $\P^k$ be the product of $\P$ with itself $k$-times. We associate each multi-index $p\in\P^k$ with the bounded operator $U[p]\colon L^2(\R^d)\to L^2(\R^d)$, defined as 
\[
U[p]f
=\begin{cases}
\ |f*f_p| &\text{if } p\in\P, \\
\ U[p_k] U[p_{k-1}] \cdots U[p_1] f &\text{if } p=(p_1,p_2,\dots,p_k)\in\P^k.
\end{cases}
\]

\begin{definition}
Given a uniform covering frame $\F$, the \emph{Fourier scattering transform} is the operator $\S_\F\colon L^2(\R^d)\to L^2(\R^d;\ell^2(\Z))$, defined as
\[
\S_\F(f)
=\{f*f_0\}\cup \{U[p]f*f_0\colon p\in\P^k,\ k=1,2,\dots\}.
\]
\end{definition}

The reason we described this operator as ``Fourier" is because each $f*f_p$ contains information about $f$ at well-localized regions in the frequency domain. There is a natural way of truncating this operator to obtain a finite sized network. Recall the following fact \cite[Proposition 2.5]{czaja2016analysis}. There exists $C_1>0$ and subsets $\{\P[m]\subset\P\colon m\geq 1\}$ such that for all $M\geq 1$,
\begin{equation}
\label{eq frame2} 
|\hat{f_0}(\xi)|^2+\sum_{p\in\P[M]} |\hat{f_p}(\xi)|^2
=\begin{cases} 
\ 1 &\text{if } \xi\in \overline{Q_{C_1M}(0)}, \\
\ 0 &\text{if } \xi\not\in \overline{Q_{C_1(M+1)}(0)}.
\end{cases}
\end{equation}
This result shows that even though $\P$ is an abstract index set, the tiling of the frequency domain, see equation (\ref{eq frame}), is done in a natural way. The definition of $\P[M]$ is used to define the following.

\begin{definition}
Given a uniform covering frame $\F$ and integers $M,K\geq 1$, the {\it truncated Fourier scattering transform} is the vector-valued operator $\S_\F[M,K]\colon L^2(\R^d)\to L^2(\R^d;\ell^2(\Z))$ defined as
\[
\S_\F(f)
=\{f*f_0\}\cup \{U[p]f*f_0\colon p\in\P[M]^k,\ k=1,2,\dots,K\}.
\]
\end{definition}
Our main results in \cite[Theorems 3.6 and 4.3]{czaja2016analysis} showed that both of the Fourier scattering transforms, under suitable assumptions, are non-expansive, are bounded above and below, and contract sufficiently small translations and diffeomorphisms. We will use precise versions of these results in Section \ref{section Fourier}.

\section{Rotational uniform covering frames}
\label{section rotational}

We construct a family of functions $\r=\r(A,B)$ that implicitly depends on two fixed parameters: $A>0$ and an integer $B\geq 1$. The functions are partially generated by finite rotation groups. For each integer $m\geq 1$, let $m^*$ denote the unique integer of the form $2^k$ such that
\begin{equation}
	\label{eq3}
	m
	\leq m^*
	<2m. 
\end{equation}
Let $R_m$ be the $2\times 2$ counter-clockwise rotation matrix
\[
R_m
=\begin{pmatrix}
\cos (B_m) &-\sin (B_m) \\
\sin (B_m) &\ \ \ \cos (B_m)
\end{pmatrix}, 
\quad\text{where}\quad
B_m = \frac{2\pi}{m^*B}.
\]
Let $G_m=G_m(B)$ be the finite rotation group generated by the following set of $d\times d$ matrices
\[
\Bigg\{
\begin{pmatrix}
R_m \\
&1 \\
& & \ddots \\
& & & \ddots \\
& & & & 1
\end{pmatrix}, \
\begin{pmatrix}
1 \\
&R_m \\
& &1 \\
& & & \ddots \\
& & & & 1
\end{pmatrix}
,\ \cdots, \
\begin{pmatrix}
1 \\
&\ddots \\
& &\ddots \\
& & & 1 \\
& & & & R_m
\end{pmatrix}
\Bigg\}.
\]
The identity element of $G_m$ is denoted $e$. We will see that $G_1$ is the most important group in this construction, so to simplify our notation, we write $G=G_1$. If $H$ is a subgroup of $G$, then we write $H \leq G$. We have the nested subgroup property,
\begin{equation}
\label{eq5}
G=G_1
\leq G_2
\leq \cdots 
\leq G_m
\leq  \cdots.
\end{equation}
It follows that for any $n\geq m$, each $r\in G_m$ is a bijection on $G_n$.

We need to introduce spherical coordinates. Let $\phi_1,\phi_2,\dots,\phi_{d-1}$ be the angular coordinates where $\phi_{d-1}\in[0,2\pi)$ and $\phi_j\in [0,\pi]$ for $j=1,2,\dots,d-2$. If $(\xi_1,\xi_2,\dots,\xi_d)$ are the Euclidean coordinates of $\xi\in\R^d$, then its spherical coordinates are $(\rho,\phi_1,\dots,\phi_{d-1})$, where $\rho=|\xi|$,
\[
\xi_d = \rho \ \prod_{j=1}^{d-1} \sin(\phi_j),
\quad\text{and}\quad
\xi_k=\rho \cos(\phi_k)\ \prod_{j=1}^{k-1} \sin (\phi_j), 
\quad\text{for}\quad k=1,2,\dots,d-1.
\]
The following proposition contains two well-known results about the existence of certain cutoff functions whose modulus squared forms a partition of unity. 

\begin{proposition}[Hern\'{a}ndez and Weiss \cite{hernandez1996first}, Chapter 1.3]
	\label{prop cutoff}
	\indent
	\begin{enumerate}[(a)]\itemsep+0.5em
		\item 
		For any $A>0$, there exists a non-negative and even $\eta_A\in C^\infty(\R)$ supported in the closed interval $[-A,A]$ such that for all $x\in\R$,
		\[
		\sum_{m\in\Z} |\eta_A(x-Am)|^2=1. 
		\] 
		\item
		For any integers $m\geq 1$ and $B\geq 1$, there exists a non-negative $\beta_{m,B}\in C^\infty(\mathbb{S}^{d-1})$, supported in the sector $\{\phi_1\colon |\phi_1|\leq B_m\}$, such that for all $\omega\in\mathbb{S}^{d-1}$,
		\[
		\sum_{r\in G_m} |\beta_{m,B}(r\omega)|^2=1.
		\]
	\end{enumerate}
\end{proposition}

We are ready to define the uniform covering frame elements in the Fourier domain and in spherical coordinates. For any $\xi\in\R^d$, we can write $\xi=\rho\omega$ where $\rho\geq 0$ and $\omega\in\s^{d-1}$. Let $f_0$ be the smooth function such that
\begin{equation}
	\label{eq10}
	\hat{f_0}(\xi)
	=\hat{f_0}(\rho\omega)
	=\eta_A(\rho).
\end{equation}
For each $m\geq 1$ and $r\in G_m$, let $f_{m,r}$ be the smooth function such that
\begin{equation}
	\label{eq11}
	\hat{f_{m,r}}(\xi)
	=\hat{f_{m,r}}(\rho\omega)
	=\eta_A(\rho-mA) \beta_{m,B}(r^{-1}\omega).
\end{equation}

The next result shows that the sequence of functions,
\[
\r
=\{f_0\}\cup\{f_{m,r}\colon m\geq 1,\ r\in G_m\},
\]
is a uniform covering frame, where the index set is
\[
\G
=\{(m,r)\colon m\geq 1,\ r\in G_m\}. 
\]

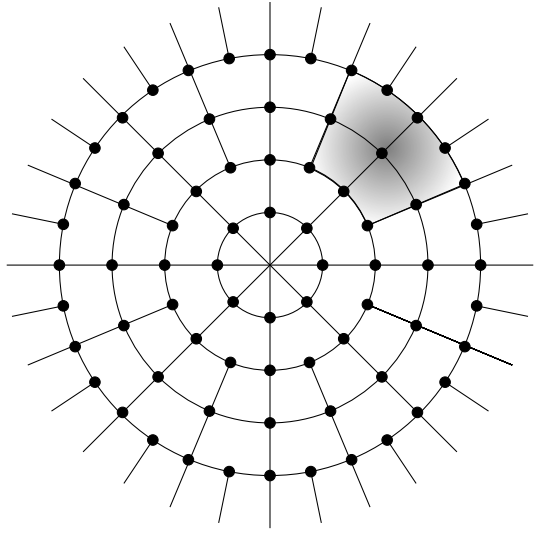
\begin{figure}[h]
	\begin{center}
		\begin{tikzpicture}[scale=.35]
		\filldraw[shading=radial,inner color=gray] 
		(3.7,1.5) arc [radius=4, start angle=22.5, delta angle=45]
		-- (3.05,7.4) arc [radius=8, start angle=67.5, delta angle=-45]
		-- cycle;
		\draw 
		(0,0) circle [radius=2] 
		(0,0) circle [radius=4]
		(0,0) circle [radius=6]
		(0,0) circle [radius=8]
		(-10,0)--(10,0)
		(0,10)--(0,-10);
		\draw
		(7.1,7.1)--(-7.1,-7.1)
		(-7.1,7.1)--(7.1,-7.1)
		(3.7,1.5)--(9.2,3.8)
		(1.5,3.7)--(3.8,9.2)
		(-3.7,1.5)--(-9.2,3.8)
		(-1.5,3.7)--(-3.8,9.2)
		(-3.7,-1.5)--(-9.2,-3.8)
		(1.5,-3.7)--(3.8,-9.2)
		(3.7,-1.5)--(9.2,-3.8)
		(-1.5,-3.7)--(-3.8,-9.2)
		(3.7,-1.5)--(9.2,-3.8)
		(7.85,1.55)--(9.8,1.95)
		(6.65,4.45)--(8.3,5.55)
		(1.55,7.85)--(1.95,9.8)
		(4.45,6.65)--(5.55,8.3)
		(3.7,-1.5)--(9.2,-3.8)
		(7.85,-1.55)--(9.8,-1.95)
		(6.65,-4.45)--(8.3,-5.55)
		(1.55,-7.85)--(1.95,-9.8)
		(4.45,-6.65)--(5.55,-8.3)
		(3.7,-1.5)--(9.2,-3.8)
		(-7.85,1.55)--(-9.8,1.95)
		(-6.65,4.45)--(-8.3,5.55)
		(-1.55,7.85)--(-1.95,9.8)
		(-4.45,6.65)--(-5.55,8.3)
		(3.7,-1.5)--(9.2,-3.8)
		(-7.85,-1.55)--(-9.8,-1.95)
		(-6.65,-4.45)--(-8.3,-5.55)
		(-1.55,-7.85)--(-1.95,-9.8)
		(-4.45,-6.65)--(-5.55,-8.3);
		\draw [fill]
		(2,0) circle [radius=.2]
		(-2,0) circle [radius=.2]
		(0,2) circle [radius=.2]
		(0,-2) circle [radius=.2]
		(1.4,1.4) circle [radius=.2]
		(-1.4,-1.4) circle [radius=.2]
		(-1.4,1.4) circle [radius=.2]
		(1.4,-1.4) circle [radius=.2]
		(4,0) circle [radius=.2]
		(-4,0) circle [radius=.2]
		(0,4) circle [radius=.2]
		(0,-4) circle [radius=.2]
		(2.8,2.8) circle [radius=.2]
		(-2.8,-2.8) circle [radius=.2]
		(-2.8,2.8) circle [radius=.2]
		(2.8,-2.8) circle [radius=.2]
		(8,0) circle [radius=.2]
		(-8,0) circle [radius=.2]
		(0,8) circle [radius=.2]
		(0,-8) circle [radius=.2]
		(5.6,5.6) circle [radius=.2]
		(-5.6,-5.6) circle [radius=.2]
		(-5.6,5.6) circle [radius=.2]
		(5.6,-5.6) circle [radius=.2]
		(3.7,1.5) circle [radius=.2]
		(5.55,2.3) circle [radius=.2]
		(7.4,3.1) circle [radius=.2]
		(-3.7,1.5) circle [radius=.2]
		(-5.55,2.3) circle [radius=.2]
		(-7.4,3.1) circle [radius=.2]
		(3.7,-1.5) circle [radius=.2]
		(5.55,-2.3) circle [radius=.2]
		(7.4,-3.1) circle [radius=.2]
		(-3.7,-1.5) circle [radius=.2]
		(-5.55,-2.3) circle [radius=.2]
		(-7.4,-3.1) circle [radius=.2]
		(1.5,3.7) circle [radius=.2]
		(2.3,5.55) circle [radius=.2]
		(3.1,7.4) circle [radius=.2]
		(-1.5,3.7) circle [radius=.2]
		(-2.3,5.55) circle [radius=.2]
		(-3.1,7.4) circle [radius=.2]
		(1.5,-3.7) circle [radius=.2]
		(2.3,-5.55) circle [radius=.2]
		(3.1,-7.4) circle [radius=.2]
		(-1.5,-3.7) circle [radius=.2]
		(-2.3,-5.55) circle [radius=.2]
		(-3.1,-7.4) circle [radius=.2]
		(6,0) circle [radius=.2]
		(-6,0) circle [radius=.2]
		(0,6) circle [radius=.2]
		(0,-6) circle [radius=.2]
		(4.25,4.25) circle [radius=.2]
		(-4.25,4.25) circle [radius=.2]
		(-4.25,-4.25) circle [radius=.2]
		(4.25,-4.25) circle [radius=.2]
		(-7.85,-1.55) circle [radius=.2]
		(-6.65,-4.45) circle [radius=.2]
		(-1.55,-7.85) circle [radius=.2]
		(-4.45,-6.65) circle [radius=.2]
		(7.85,-1.55) circle [radius=.2]
		(6.65,-4.45) circle [radius=.2]
		(1.55,-7.85) circle [radius=.2]
		(4.45,-6.65) circle [radius=.2]
		(-7.85,1.55) circle [radius=.2]
		(-6.65,4.45) circle [radius=.2]
		(-1.55,7.85) circle [radius=.2]
		(-4.45,6.65) circle [radius=.2]
		(7.85,1.55) circle [radius=.2]
		(6.65,4.45) circle [radius=.2]
		(1.55,7.85) circle [radius=.2]
		(4.45,6.65) circle [radius=.2];
	\end{tikzpicture}
	
	\caption{Let $d=2$ and $G$ be the group of rotations by angle $2\pi/8$. The black dots are elements of $\G$, when embedded in $\R^2$, for the first four uniform Fourier scales, and the shaded gray region is the support of $(f_{3,2\pi/8})^\wedge$.}
	\label{fig 2}
\end{center}
\end{figure}

\begin{theorem}
	\label{thm rot}
	Let $\r=\{f_0\}\cup\{f_{m,r}\colon (m,r)\in\G\}$ be the sequence of functions defined above. Then, $\r$ is a uniform covering frame for $L^2(\R^d)$.
\end{theorem}

\begin{proof}
	By construction, $\hat{f_0}$ and $\hat{f_p}$ are supported in a compact and connected sets. By Proposition \ref{prop cutoff}, we have $\hat{f_0}(0)=\eta_A(0)=1$. Since $f_0$ and $f_p$ are Schwartz functions, they belong to $L^1(\R^d)\cap L^2(\R^d)\cap C^1(\R^d)$. 
	
	We check that the frame condition holds. Using (\ref{eq10}), (\ref{eq11}), and Proposition \ref{prop cutoff}, we see that
	\begin{align*}
	&|\hat{f_0}(\xi)|^2+\sum_{m=1}^\infty\sum_{r\in G_m}|\hat{f_{m,r}}(\xi)|^2 \\
	&\quad=|\eta_A(\rho)|^2 + \sum_{m=1}^\infty |\eta_A(\rho-Am)|^2 \sum_{r\in G_m}|\beta_{m,B}(r^{-1}\omega)|^2
	=1,
	\end{align*}
	for all $\xi\in\R^d$. See Figure \ref{fig 2} for a visualization of $\G$ and the tiling properties of $\r$.
	
	It remains check that the uniform covering property holds. For each $m\geq 1$ and $r\in G_m$, since $\supp(\hat{f_{m,r}})$ is a rotation of $\supp(\hat{f_{m,e}})$, it suffices to check the uniform covering property for the family of sets $\{\supp(\hat{f_{m,e}})\colon m\geq 1\}$. Further, it suffices to check the uniform covering property for the subset $\{\supp(\hat{f_{m,e}})\colon m^*B\geq 4\}$, since the complement of this set is $\{\supp(\hat{f_{m,e}})\colon m^*B<4\}$, which has finite cardinality. 
	
	From here onwards, we assume $m^* B\geq 4$, or equivalently, $B_m\leq \pi/2$. Observe that $\hat{f_{m,e}}$ is supported in the wedge,
	\begin{equation}
		\label{eq W}
		W_m = \Big\{(\rho,\phi_1,\phi_2,\dots,\phi_{d-1})\colon |\rho-Am|\leq A,\ |\phi_1|\leq B_m\Big\}.
	\end{equation}
	To show that $\{W_m\colon m^*B\geq 4\}$ satisfies the uniform covering property, it suffices to show that the maximum distance between any two points in $W_m$ is bounded uniformly in $m$. 
	
	Let $\xi,\zeta\in W_m$ and let $(\rho,\phi_1,\phi_2,\dots,\phi_{d-1})$ and $(\gamma,\theta_1,\theta_2,\dots,\theta_{d-1})$ be their spherical coordinates, respectively. We have
	\[
	|\xi_1-\zeta_1|
	=|\rho\cos(\phi_1)-\gamma\cos(\theta_1)|. 
	\]
	Since $|\phi_1|\leq B_m\leq \pi/2$ and similarly for $\theta_1$, we see that
	\[
	|\xi_1-\zeta_1|
	\leq A(m+1)-A(m-1)\cos(B_m).
	\]
	See Figure \ref{fig 3} for an illustration of this inequality. Using the standard trigonometric inequality, $|1-\cos(t)|\leq t^2/2$ for all $t\in\R$, we see that
	\begin{equation}
		\label{eq4}
		|\xi_1-\zeta_1|
		\leq 2A+\frac{Am}{2}\(\frac{2\pi}{m^*B}\)^2
		= 2A+\frac{2\pi^2 A}{B^2}\frac{m}{(m^*)^2}. 
	\end{equation}
	For $k=2,3,\dots,d-1$, we have
	\[
	|\xi_k-\zeta_k|
	= \Big|\rho \cos(\phi_k) \prod_{j=1}^{k-1} \sin(\phi_j) - \gamma \cos(\theta_k)\prod_{j=1}^{k-1} \sin(\theta_j)\Big|
	\leq \rho |\sin(\phi_1)|+\gamma|\sin(\theta_1)|. 
	\]
	Using that $\rho\leq A(m+1)$, the trigonometric inequality $|\sin(\phi_1)|\leq |\phi_1|\leq B_m$, and similarly for $|\sin(\theta_1)|$, we have
	\begin{equation}
		\label{eq7}
		|\xi_k-\zeta_k|
		\leq 2 A(m+1)B_m
		\leq \frac{4\pi A}{B}\frac{m+1}{m^*}.
	\end{equation}
	The same argument shows that
	\begin{equation}
		\label{eq8}
		|\xi_d-\zeta_d|
		\leq \rho |\sin(\phi_1)|+\gamma|\sin(\theta_1)|
		\leq \frac{4\pi A}{B}\frac{m+1}{m^*}.
	\end{equation}
	The inequalities (\ref{eq3}), (\ref{eq4}), (\ref{eq7}), and (\ref{eq8}) imply the Euclidean distance between any two points in $W_m$ is bounded uniformly in $m$, which completes the proof. 
	
	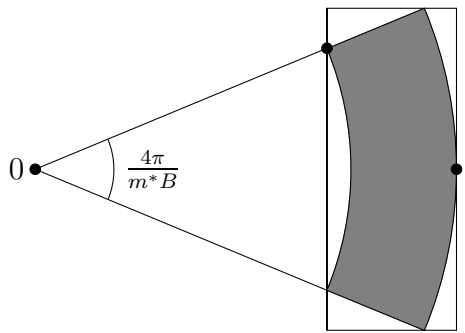
\begin{figure}[h]
		\begin{center}
			\begin{tikzpicture}[scale=.7]
			\draw [fill, color=gray]
			(5.54,-2.30) arc [radius=6, start angle=337.5, delta angle=45]
			-- (7.39,3.06) arc [radius=8, start angle=22.5, delta angle=-45]
			-- cycle;
			\draw 
			(5.54,-2.30) arc [radius=6, start angle=337.5, delta angle=45]
			-- (7.39,3.06) arc [radius=8, start angle=22.5, delta angle=-45]
			-- cycle;
			\draw [fill]
			(0,0) circle [radius=.1];
			\draw 
			(1.38,-0.58) arc [radius=1.5, start angle=337.5, delta angle=45];
			\draw 
			(5.54,-3.06) 
			-- (5.54,3.06)
			-- (8,3.06)
			-- (8,-3.06)
			-- (5.54,-3.06)
			-- (5.54,3.06);
			\draw 
			(0,0) -- (5.54,-2.30) 
			(0,0) -- (5.54,2.30); 
			\draw
			(0,0) node[left]{$0$};	
			\draw
			(1.5,0) node[right]{$\frac{4\pi}{m^*B}$};
			\draw [fill]
			(5.54,2.30) circle [radius=.1];
			\draw [fill]
			(8,0) circle [radius=.1];
			\end{tikzpicture}
		\end{center}
		
		\caption{An illustration of inequality (\ref{eq4}). The shaded region represents all possible values of $(\rho,\phi_1)$ and $(\gamma,\theta_1)$, where $A(m-1)\leq\rho\leq A(m+1)$ and $|\theta_1|\leq B_m$, and likewise for $\gamma$ and $\phi_1$. The maximum horizontal distance between any points in this wedge is the width of the rectangle, and this distance is attained by the points $(A(m+1),0)$ and $(A(m-1)\cos(2\pi/(m^*B)),A(m-1)\sin(2\pi/(m^*B)))$, located at the two black dots on the perimeter of the wedge.}
		\label{fig 3}
		
	\end{figure}
\end{proof}

\begin{definition}
	Fix $A>0$ and an integer $B\geq 1$ and let 
	\[
	\r=\r(A,B)=\{f_0\}\cup\{f_{m,r}\colon (m,r)\in\G\}
	\]
	be the sequence of functions constructed in this section. We call $\r$ a \emph{rotational uniform covering frame}. 
\end{definition}

\begin{remark}
	Observe that $\hat{f_{m,r}}(\xi)=\hat{f_{m,e}}(r^{-1}\xi)$, so $f_{m,r}$ as a rotation of $f_{m,e}$. We also have $\hat{f_{m+1,e}}(\xi)=\hat{f_{1,e}}(\xi-mA)$, so each $f_{m,e}$ is a modulation of $f_{1,e}$. For this reason, we say rotational uniform covering frames are generated by rotations and modulations.
\end{remark}

\section{Rotational Fourier scattering transform}
\label{section Fourier}

Let $\r$ be a rotational uniform covering frame and recall that its index set is $\G=\{(m,r)\colon m\geq 1,\ r\in G_m\}$ with $G=G_1$. In order to derive a time-frequency scattering transform that is invariant under rotations $G$, we carefully study its group action on $\r$. This is carried out in the subsequent steps. 

\begin{enumerate} 
	\itemsep +1em
	\item 
	We define the following left and right group actions of $G$ on $\G$. For each $r\in G$ and $(m,s)\in\G$, let
	\[
	r(m,s)=(m,rs)
	\quad\text{and}\quad
	(m,s)r=(m,sr).
	\]
	Both operations are well-defined in view of the nested subgroup property (\ref{eq5}). 
	
	For any $k\geq 1$, we define the left and right group actions of $G$ on $\G^k$ by extending the above definitions. Indeed, each $r\in G$ acts on $p\in\G^k$ on the left according to
	\[
	r(p_1,p_2,\dots,p_k)
	=(rp_1,rp_2,\dots,rp_k),
	\]
	and similarly for its right action. 
	
	\item
	The previous observations lead us to the following decomposition. We saw that $G$ is a group action on $\G^k$. Then, there exists a set $\Q_k\subset\G^k$ such that we have the disjoint union
	\begin{equation}
	\label{eq Q}
	\G^k=\bigcup_{r\in G} r\Q_k.
	\end{equation}
	Explicitly, we have $\Q_k=\G_0 \times \G^{k-1}$, where $\G_0=\{(m,r)\colon m\geq 1,\ r\in G_m/G\}$ and $G_m/G$ is the quotient group, $G_m$ modulo $G$. 
	
	\item
	We can consider the action of $G$ on $\r$. Each $r\in G$ acts on $L^2(\R^d)$ by the rule 
	\[
	r\colon f(x)\mapsto f_r(x)=f(rx).
	\]
	Using the identity (\ref{eq11}), for any $r\in G$ and $(m,s)\in\G$, 
	\begin{equation}
		\label{eq6}
		(f_{m,s})_r(x)
		=f_{m,s}(rx)
		=f_{m,r^{-1}s}(x). 
	\end{equation}
	This defines a group action of $G$ on $\r$. Further, this calculation shows that the invariant subsets of $\r$ under $G$ are $\{f_0\}$ and $\{f_{m,s}\colon s\in G_m\}$, for each $m\geq 1$.
	
	\item
	For any $r\in G$ and $(m,s)\in\G$, by identity (\ref{eq6}) and a change of variables, 
	\[
	(f_r*f_{m,s})(x)
	=(f*(f_{m,s})_{r^{-1}})(rx)
	=(f*f_{m,rs})_r(x).
	\]
	By iterating this identity, for all integers $k\geq 1$ and $p\in\G^k$, we obtain
	\[
	(U[p]f_r)(x)
	=(U[rp]f)_r(x). 
	\]
	Since definition (\ref{eq10}) implies that $f_0$ is rotationally invariant, we see that
	\begin{equation}
	\label{eq1} 
	(U[p]f_r*f_0)(x)
	=((U[rp]f)_r*f_0)(x)
	=(U[rp]f*f_0)_r(x).
	\end{equation}
	
	\item
	We now carry out the same steps but for specific finite subsets of $\G^k$ and $\Q_k$. Fix integers $k\geq 1$ and $M\geq 1$. We define the finite set
	\[
	\G[M]
	=\{(m,r)\colon 1\leq m\leq M,\ r\in G_m\}.
	\]
	By construction of $\r$, we have
	\begin{equation}
	\label{eq2}
	|\hat{f_0}(\xi)|^2+\sum_{p\in\G[M]} |\hat{f_p}(\xi)|^2
	=
	\begin{cases}
	\ 1 &\text{if } |\xi|\leq AM, \\
	\ 0 &\text{if } |\xi|\geq A(M+1).
	\end{cases}
	\end{equation}
	Let $\G[M]^k$ be the product of $\G[M]$ with itself $k$ times. Using the same definition as before, we see that $G$ is a group action on $\G[M]^k$. Similar to before, there exists a finite set $\Q[M,K]\subset\G[M]^k$ such that we have the disjoint union 
	\begin{equation}
	\label{eq Q2}
	\G[M]^k=\bigcup_{r\in G} r\Q[M,k].
	\end{equation}
	Explicitly, we have $\Q[M,k]=\G_0[M] \times \G[M]^k$ where $\G_0[M]=\{(m,r)\colon 1\leq m\leq M,\ r\in G_m/G\}$. For all integers $k\geq 1$ and $p\in\G[M]^k$, the same argument as before shows that
	\begin{equation}
	\label{eq0} 
	(U[p]f_r*f_0)(x)
	=(U[rp]f*f_0)_r(x).
	\end{equation} 
\end{enumerate}

\begin{definition}
	Given a rotational uniform covering frame $\r$, the associated \emph{rotational Fourier scattering transform} $\S^\r$ is formally defined as
	\begin{align*}
	&\S^\r(f) \\
	&=\Big\{ |G|^{-1/2} \big\|(f*f_0)_r \big\|_{\ell^2_r(G)} \Big\} \cup \Big\{ \big\|(U[rq]f*f_0)_r \big\|_{\ell^2_r(G)} \colon q\in \Q_k,\ k\geq 1\Big\} \\
	&=\Big\{ |G|^{-1/2}\(\sum_{r\in G} |(f*f_0)_r|^2\)^{1/2} \Big\}
	\cup \Big\{ \(\sum_{r\in G} |(U[rq]f*f_0)_r|^2\)^{1/2} \colon q\in \Q_k,\ k\geq 1 \Big\}.
	\end{align*}
\end{definition}

\begin{definition}
	Given a rotational uniform covering frame $\r$ and integers $M,K\geq 1$, the associated \emph{truncated rotational Fourier scattering transform} $\S^\r[M,K]$ is formally defined as 
	\begin{align*}
	&\S^\r[M,K](f) \\
	&=\Big\{ |G|^{-1/2} \big\|(f*f_0)_r \big\|_{\ell^2_r(G)} \Big\}\cup \Big\{ \big\|(U[rq]f*f_0)_r \big\|_{\ell^2_r(G)} \colon q\in \Q[M,k],\ 1\leq k\leq K\Big\}.
	\end{align*}
\end{definition}

\begin{remark}
	We explain our choice of notation. Throughout this paper, we used $\F$ to denote a generic uniform covering frame whereas we used $\r$ to mean the rotational variant. By assumption, $\S_\F$ denotes the Fourier scattering transform and the subscript emphasizes that it depends on a fixed $\F$. Similarly, the super-script in $\S^\r$ indicates that the rotational Fourier scattering transform depends on a fixed $\r$. By Theorem \ref{thm rot}, $\r$ is a perfectly valid uniform covering frame, so it can be used as the underlying frame in the Fourier scattering transform, and this operator is denoted $\S_\r$. However, we emphasize that while both $\S_\r$ and $\S^\r$ use the same frame $\r$, they are \emph{distinct} operators, and likewise for their truncations. This can be easily seen from their definitions or by comparing their network structures. While the sub-script and super-script notation might seem odd, we use them to differentiate between these two operators. In the next two propositions, we derive their quantitative relationship.
\end{remark}

\begin{proposition}
	\label{prop1}
	Let $\r=\{f_0\}\cup \{f_{m,r}\colon (m,r)\in\G\}$ be a rotational uniform covering frame. For all $f,g\in L^2(\R^d)$, we have
	\begin{align*}
	\|\S^\r(f)\|_{L^2\ell^2}&= \|\S_\r(f)\|_{L^2\ell^2}, \\
	\|\S^\r(f)-\S^\r(g)\|_{L^2\ell^2}&\leq \|\S_\r(f)-\S_\r(g)\|_{L^2\ell^2}.
	\end{align*}
\end{proposition}

\begin{proof}
	To prove the equality, we use the decomposition (\ref{eq Q}) to obtain
	\begin{align*}
	\|\S^\r(f)\|_{L^2\ell^2}^2
	&=\frac{1}{|G|}\sum_{r\in G} \|(f*f_0)_r\|_{L^2}^2 + \sum_{k=1}^\infty\sum_{q\in\Q_k}\sum_{r\in G} \|(U[rq]f*f_0)_r\|_{L^2}^2 \\
	&=\|f*f_0\|_{L^2}^2 + \sum_{k=1}^\infty\sum_{q\in\Q_k}\sum_{r\in G} \|U[rq]f*f_0\|_{L^2}^2 \\
	&=\|f*f_0\|_{L^2}^2 + \sum_{k=1}^\infty \sum_{p\in\G^k}\|U[p]f*f_0\|_{L^2}^2 
	=\|\S_\r(f)\|_{L^2\ell^2}^2. 
	\end{align*}
	To prove the inequality, we apply the reverse triangle inequality for the $\ell^2(G)$ norm in the definition of $\S^\r$ and the decomposition (\ref{eq Q}),
	\begin{align*}
	\|\S^\r(f)-\S^\r(g)\|_{L^2\ell^2}^2
	&\leq \frac{1}{|G|} \sum_{r\in G}\|(f*f_0)_r-(g*f_0)_r\|_{L^2}^2 \\
	&\quad\quad +\sum_{k=1}^\infty \sum_{q\in\Q_k}\sum_{r\in G} \|(U[rq]f*f_0)_r-(U[rq]g*f_0)_r\|_{L^2}^2 \\
	&= \|f*f_0-g*f_0\|_{L^2}^2 + \sum_{k=1}^\infty \sum_{q\in\Q_k}\sum_{r\in G} \|U[rq]f*f_0-U[rq]g*f_0\|_{L^2}^2 \\
	&= \|f*f_0-g*f_0\|_{L^2}^2 + \sum_{k=1}^\infty \sum_{p\in\G^k} \|U[p]f*f_0-U[p]g*f_0\|_{L^2}^2 \\
	& =\|\S_\r(f)-\S_\r(g)\|_{L^2\ell^2}^2.
	\end{align*}
\end{proof}

\begin{proposition}
	\label{prop2}
	Let $\r=\{f_0\}\cup \{f_{m,r}\colon (m,r)\in\G\}$ be a rotational uniform covering frame. There exists $C_2>0$ such that for any integers $M,K\geq 1$ and integer $N\leq C_2M$, 
	\[
	\|\S^\r[M,K](f)\|_{L^2\ell^2}
	\geq \|\S_\r[N,K](f)\|_{L^2\ell^2}.
	\] 	
\end{proposition}

\begin{proof}
	By comparing the identities (\ref{eq frame2}) and (\ref{eq2}), we see that there exists a constant $C_2>0$ such that whenever $N\leq C_2M$, we have $\P[N]^k\subset \G[M]^k$. This fact and the disjoint decomposition \eqref{eq Q2} imply
	\begin{align*}
	\|\S^\r[M,K](f)\|_{L^2\ell^2}^2
	&=\frac{1}{|G|}\sum_{r\in G} \|(f*f_0)_r\|_{L^2}^2 +\sum_{k=1}^K\sum_{q\in\Q[M,k]}\sum_{r\in G}\|(U[rq]f*f_0)_r\|_{L^2}^2 \\
	&=\|f*f_0\|_{L^2}^2+ \sum_{k=1}^K\sum_{p\in\G[M]^k} \|U[p]f*f_0\|_{L^2}^2 \\
	&\geq \|f*f_0\|_{L^2}^2+ \sum_{k=1}^K\sum_{p\in\P[N]^k} \|U[p]f*f_0\|_{L^2}^2
	=\|\S_\r[N,K](f)\|_{L^2\ell^2}. 
	\end{align*}
\end{proof}

We are ready to state and prove the main theorem, which shows that $\S^\r$ and $\S^\r[M,K]$ are effective feature extractors. We already did most of the work in the construction of $\S^\r$ and $\S^\r[M,K]$, and in proving Propositions \ref{prop1} and \ref{prop2}. The basic strategy is to quantitatively relate $\S^\r$ to $\S_\r$ and then use known results about the latter.

\begin{theorem}
	\label{thm1}
	Let $\r=\{f_0\}\cup\{f_{m,r}\colon (m,r)\in\G\}$ be a rotational uniform covering frame.   
	\begin{enumerate}[(a)]
		\itemsep+1em
		
		\item
		\underline{$G$-invariance}: For all integers $M,K\geq 1$, $f\in L^2(\R^d)$, and $r\in G$,
		\[
		\S^\r(f_r) = \S^\r(f)
		\quad\text{and}\quad
		\S^\r[M,K](f_r)=\S^\r[M,K](f). 
		\]
		
		\item 
		\underline{Upper bound}: For all integers $M,K\geq 1$ and $f\in L^2(\R^d)$, 
		\[
		\|\S^\r[M,K](f)\|_{L^2\ell^2}
		\leq\|\S^\r(f)\|_{L^2\ell^2}
		=\|f\|_{L^2}.
		\]
		
		\item
		\underline{Lower bound}: Let $\epsilon\in [0,1)$ and $R>0$. For sufficiently large $M,K\geq 1$, there exists $C\in (0,1)$ depending on $\r,\epsilon,M,K$, such that for all $(\epsilon,R)$ band-limited functions $f\in L^2(\R^d)$,  
		\[
		\|\S^\r[M,K](f)\|_{L^2\ell^2} 
		\geq C\|f\|_{L^2}.
		\]
		
		\item
		\underline{Non-expansiveness}: For all integers $M,K\geq 1$ and $f,g\in L^2(\R^d)$, 
		\[
		\|\S^\r[M,K](f)-\S^\r[M,K](g)\|_{L^2\ell^2}
		\leq \|\S^\r(f)-\S^\r(g)\|_{L^2\ell^2}
		\leq \|f-g\|_{L^2}.
		\]
		
		\item
		\underline{Translation contraction}: There exists $C>0$ depending only on $\r$ such that for all integers $M,K\geq 1$, $f\in L^2(\R^d)$, and $y\in\R^d$,
		\begin{align*}
			&\|\S^\r[M,K](T_y f)-\S^\r[M,K](f)\|_{L^2\ell^2} \\
			&\quad\quad\leq \|\S^\r(T_y f)-\S^\r(f)\|_{L^2\ell^2} 
			\leq C|y|\|\nabla f_0\|_{L^1} \|f\|_{L^2}.
		\end{align*}
		
		\item
		\underline{Additive diffeomorphism contraction}: There exists a constant $C>0$ such that for all $\epsilon\in [0,1)$, $R>0$, $(\epsilon,R)$ band-limited $f\in L^2(\R^2)$, and $\tau\in C^1(\R^2;\R^2)$ with $\|\nabla\tau\|_{L^\infty}\leq 1/4$,
		\begin{align*}
			&\|\S^\r[M,K](T_\tau f)-\S^\r[M,K](f)\|_{L^2\ell^2} \\
			&\quad\quad\leq \|\S^\r(T_\tau f)-\S^\r(f)\|_{L^2\ell^2} 
			\leq C(R\|\tau\|_{L^\infty}+\epsilon)\|f\|_{L^2}.
		\end{align*}
	\end{enumerate}
\end{theorem}

\begin{proof}
	\indent
	\begin{enumerate}[(a)]\itemsep+.5em
		\item 
		For the first term in $\S^\r(f_r)$ and $\S^\r[M,K](f_r)$, we use that $f_0$ is rotationally invariant to obtain 
		\[
		\sum_{s\in G} |(f_r*f_0)_s|^2
		=\sum_{s\in G} |(f*f_0)_{rs}|^2
		=\sum_{s\in G} |(f*f_0)_s|^2.
		\]
		For the remaining terms in $\S^\r(f_r)$, fix an integer $k\geq 1$ and $q\in \Q_k$. By identity (\ref{eq1}) and re-indexing the following sum, we have
		\[
		\sum_{s\in G} |(U[sq]f_r*f_0)_s|^2
		=\sum_{s\in G} |(U[rsq]f*f_0)_{rs}|^2
		=\sum_{s\in G} |(U[sq]f*f_0)_s|^2.
		\]
		This proves that $\S^\r(f_r)=\S^\r(f)$. For the remaining terms in $\S^\r[M,K](f_r)$, fix an integer $1\leq k\leq K$ and $q\in\Q[M,k]$. By identity (\ref{eq0}) and re-indexing the following sum, we have
		\[
		\sum_{s\in G} |(U[sq]f_r*f_0)_s|^2
		=\sum_{s\in G} |(U[rsq]f*f_0)_{rs}|^2
		=\sum_{s\in G} |(U[sq]f*f_0)_s|^2.
		\]
		This proves that $\S^\r[M,K](f_r)=\S^\r[M,K](f)$ and $\S^\r(f_r)=\S^\r(f)$.
		
		\item
		By definition and Proposition \ref{prop1}, we have
		\[
		\|\S^\r[M,K](f)\|_{L^2\ell^2}
		\leq \|\S^\r(f)\|_{L^2\ell^2}
		=\|\S_\r(f)\|_{L^2\ell^2}.
		\]
		We apply \cite[Theorem 3.6a]{czaja2016analysis} to complete the proof.
		
		\item
		By Proposition \ref{prop2}, there exists $C_2>0$ such that for all $N\leq C_2M$, we have 
		\[
		\|\S^\r[M,K](f)\|_{L^2\ell^2} 
		\geq \|\S_\r[N,K](f)\|_{L^2\ell^2}.
		\]
		Using the lower bound for $\S_\r[N,K]$, see \cite[Theorem 4.5b]{czaja2016analysis}, completes the proof. 
		
		\item
		By definition and Proposition \ref{prop1}, we have
		\[
		\|\S^\r[M,K](f)-\S^\r[M,K](g)\|_{L^2\ell^2}
		\leq \|\S_\r(f)-\S_\r(g)\|_{L^2\ell^2}.
		\]
		Since $\S_\r$ is non-expansive, see \cite[Theorem 3.6b]{czaja2016analysis}, the result follows.
		
		\item
		By definition and Proposition \ref{prop1}, 
		\[
		\|\S^\r[M,K](T_yf)-\S^\r[M,K](f)\|_{L^2\ell^2}
		\leq \|\S_\r(T_yf)-\S_\r(f)\|_{L^2\ell^2}.
		\]
		The rest follows by using the translation estimate for $\S_\r$, see \cite[Theorem 3.6c]{czaja2016analysis}.
		
		\item
		Again, we use the definition and Proposition \ref{prop1} to deduce
		\[
		\|\S^\r[M,K](T_\tau f)-\S^\r[M,K](f)\|_{L^2\ell^2}
		\leq \|\S_\r(T_\tau f)-\S_\r(f)\|_{L^2\ell^2}.
		\]
		We use the diffeomorphism estimate, \cite[Theorem 3.6d]{czaja2016analysis}, to complete the proof. 
	\end{enumerate}
\end{proof}

\begin{remark}
	We saw in the proof of Theorem \ref{thm1} that if we replaced the $\ell^2(G)$ norm in the definition of $\S^\r$ and $\S^\r[M,K]$ with the $\ell^p(G)$ norm for any $1\leq p<\infty$, then the resulting operators would still be $G$-invariant. Additionally, since $G$ is a finite set, all $\ell^p(G)$ norms are equivalent up to constants depending on $p$ and $|G|$, so our estimates carry generalize to the $\ell^p(G)$ case. We chose the $\ell^2(G)$ norm in the definitions of $\S^\r$ and $\S^\r[M,K]$ because it is more natural to work with a Hilbert space as opposed to a Banach space. 
\end{remark}

\section{Discrete and finite uniform covering frames}
\label{section finite}

For other data analysis and signal processing applications, all functions must be converted into arrays. We show how to construct discrete and finite versions of uniform covering frames. Throughout this section, let $\F$ be a uniform covering frame with index set $\P$, and note that $\F$ does not necessarily have to be a rotational uniform covering frame.

We introduce the following standard notation to simplify the formulas in this section. For multi-integers $n,m\in\Z^d$, we write $m\leq n$ to mean $m_j\leq n_j$ for all $1\leq j\leq d$. If $n>0$, let
\[
\vol(n) = \prod_{j=1}^d n_j
\quad\text{and}\quad
\frac{m}{n} = \(\frac{m_1}{n_1},\frac{m_2}{n_2},\dots,\frac{m_d}{n_d}\).
\]

\subsection{Discrete frames}

To discretize uniform covering frames, we take advantage of the fact that $f_p$ is band-limited for each $p\in\P$ and use ideas from classical sampling theory. By definition of a uniform covering frame, the diameter of $\supp(\hat{f_p})$ is bounded uniformly in $p$, see \cite[Proposition 2.5]{czaja2016analysis}. Let $S_p=(S_{p,1},S_{p,2}\dots,S_{p,d})$ be the side lengths of the smallest rectangle that contains $\supp(\hat{f_p})$. Similarly, let $S_0=(S_{0,1},S_{0,2},\dots,S_{0,d})$ be the side lengths of the smallest rectangle that contains $\supp(\hat{f_0})$. Then, for $n\in\Z^d$, let
\[
f_{0,n}(x)= \vol(S_0)^{-1/2} f_0 \(x+\frac{n}{S_0}\) 
\quad\text{and}\quad 
f_{p,n}(x)= \vol(S_p)^{-1/2} f_p\(x+\frac{n}{S_p}\)
\]
\begin{definition}
	We call the sequence of functions 
	\[
	\F_{\text{discrete}}
	=\{f_{0,n}\colon n\in\Z^d\}\cup \{f_{p,n}\colon p\in\P,\ n\in\Z^d\}	
	\]
	a \emph{discrete rotational uniform covering frame}.
\end{definition}

\begin{proposition}
	$\F_{\text{discrete}}$ is a Parseval frame for $L^2(\R^d)$: for all $f\in L^2(\R^d)$, 
	\[
	\sum_{n\in\Z^d} |\<f,f_{0,n}\>|^2+\sum_{p\in\P} |\<f,f_{p,n}\>|^2
	=\|f\|_{L^2}^2. 
	\] 
\end{proposition}

\begin{proof}
	We will repeatedly use the well-known fact that for any multi-integer $S\in\Z^d$ with $S>0$, the sequence of functions,
	\[
	\{\vol(S)^{-1/2}e^{-2\pi i\xi\cdot n/S}\colon n\in\Z^d\},
	\]
	is an orthonormal basis for $L^2(\R^d)$ functions that are supported in any rectangle of side length $S$. 
	
	By Parsvel, we have
	\[
	\<f,f_{p,n}\>
	=\vol(S_p)^{-1/2}\int_{\R^d} \hat f(\xi) \overline{\hat{f_p}(\xi)} e^{-2\pi i\xi\cdot n/S_p}\ d\xi.
	\]
	Since $\hat{f_p}$ is supported in a translate of a rectangle with side length $S$, we see that
	\begin{equation}
	\label{eq13}
	\sum_{n\in\Z^d} |\<f,f_{p,n}\>|^2
	=\int_{\R^d} |\hat f(\xi)|^2 |\hat{f_p}(\xi)|^2\ d\xi.
	\end{equation}
	We use the same argument to handle terms that involve the inner product of $f$ with $f_{0,n}$. Doing so, we obtain
	\begin{equation}
	\label{eq12}
	\sum_{n\in\Z^d} |\<f,f_{0,n}\>|^2
	=\int_{\R^d} |\hat f(\xi)|^2 |\hat{f_0}(\xi)|^2 \ d\xi. 
	\end{equation}	
	Combining equations (\ref{eq13}), (\ref{eq12}) and (\ref{eq frame}), we conclude that
	\begin{align*}
	&\sum_{n\in\Z^d} |\<f,f_{0,n}\>|^2+\sum_{p\in\P} \sum_{n\in\Z^d} |\<f,f_{p,n}\>|^2 \\
	&\quad =\int_{\R^d} |\hat f(\xi)|^2 |\hat{f_0}(\xi)|^2 \ d\xi + \sum_{p\in\P} \int_{\R^d} |\hat f(\xi)|^2 |\hat{f_{m,r}}(\xi)|^2\ d\xi =\int_{\R^d} |\hat f(\xi)|^2 \ d\xi. 
	\end{align*}
\end{proof}

\subsection{Finite frames}

To define the finite uniform covering frames, we will replace the Euclidean Fourier transforms with the $d$-dimensional discrete Fourier transform (DFT), and replace $\hat{f_p}$ with its samples on a lattice. 

Suppose $F$ is an array of size $N$, which we want to decompose using the finite uniform covering frames. Let $\hat{F}(\xi)$ be the $d$-dimensional Fourier series of $F$ evaluated at $\xi\in\R^d$,
\[
\hat F(\xi)
=\sum_{1\leq n\leq N} F(n) e^{-2\pi i\xi\cdot n/N}. 
\]
Its DFT is $N$-periodic and we view them as samples of its Fourier series evaluated at the points $m_j=-\overline N_j,-\overline N_j+1,\dots,\overline N_j$, where $\overline N\in\R^d$ is defined as $\overline N_j=(N_j-1)/2$. Call this set of points $\supp(\hat F)$.

Let $\P_N$ be the finite subset of $\P$ such that $p\in\P_N$ if and only if $\supp(\hat{f_p})$ has nontrivial intersection with the rectangle
\[
[-\overline N_1,\overline N_1]\times[-\overline N_2,\overline N_2]\times \cdots\times [-\overline N_d,\overline N_d].
\]
Let $S_p-1\in\Z^d$ be the side length of the smallest rectangle containing $\supp(\hat{f_p})\cap\supp(\hat F)$, and let $S_0$ be defined analogously. For each $p\in\P_N$ and $0\leq n\leq S_p-1$, let $F_{p,n}$ be an array of size $N$, which we define according to its DFT, 
\begin{align*}
\hat{F_{p,n}}(m)&=F_p(m)E_{p,n}(m), \\
F_p(m)&=\hat{f_p}(m), \\
E_{p,n}(m)&= \vol(S_p)^{-1/2} e^{-2\pi im\cdot n/S_p}.
\end{align*}
Similarly, we define the array $F_{0,n}$ by
\begin{align*}
\hat{F_{0,n}}(m)&=F_0(m)E_{0,n}(m), \\
F_0(m)&=\hat{f_0}(m), \\
E_{0,n}(m)&= \vol(S_0)^{-1/2} e^{-2\pi im\cdot n/S_0}.
\end{align*}

If $F$ and $G$ are arrays of size $N>0$, their Frobenius inner product is 
\[
\<F,G\>=\sum_{1\leq n\leq N} F(n)\overline{G(n)}
\]
and the Frobenius norm is $\|F\|_2=\sqrt{\<F,F\>}$. 

\begin{definition}
	Given $N=(N_1,N_2,\dots,N_d)>0$, we call the set
	\[
	\F_{\text{finite}}
	=\{F_{0,n}\colon 0\leq n\leq S_0-1\}\cup \{F_{p,n}\colon p\in\P_N, 0\leq n\leq S_p-1\},
	\]
	a \emph{finite uniform covering frame}. 
\end{definition}

\begin{proposition}
	For any $N=(N_1,N_2,\dots,N_d)>0$, the set $\F_{\text{finite}}$ is a finite Parseval frame for arrays of size $N$: for all arrays $F$ of size $N$,
	\[
	\sum_{0\leq n\leq S_0-1}|\<F,F_{0,n}\>|^2 +\sum_{p\in\P_N}\sum_{0\leq n\leq S_p-1} |\<F,F_{p,n}\>|^2
	=\|F\|_2^2. 
	\]
\end{proposition}

\begin{proof}
	By Parseval, we have
	\[
	|\<F,F_{p,n}\>|^2
	=\frac{1}{\vol(N)}\ \Big|\sum_{-\overline N\leq m\leq \overline N} \hat{F}(m) \overline{F_p(m)} \overline{E_{p,n}(m)} \Big|^2.
	\]
	Since $\{E_{p,n}\colon 0\leq n\leq S_p-1\}$ is an orthonormal basis for the support of $F_p$ with respect to the inner product $\<\cdot,\cdot\>$, we have
	\begin{equation}
	\label{eq17}
	\sum_{0\leq n\leq S_p-1} |\<F,F_{p,n}\>|^2
	=\frac{1}{\vol(N)} \sum_{-\overline N\leq m\leq \overline N} |\hat{F}(m)|^2 |F_p(m)|^2.
	\end{equation}
	Repeating the same argument shows that
	\begin{equation}
	\label{eq18}
	\sum_{0\leq n\leq S_0-1} |\<F,F_{0,n}\>|^2
	=\frac{1}{\vol(N)} \sum_{-\overline N\leq m\leq \overline N} |\hat{F}(m)|^2 |F_0(m)|^2.
	\end{equation}
	Combining the equations (\ref{eq17}), (\ref{eq18}), (\ref{eq frame}) and recalling that $F_p$ are samples of $\hat{f_p}$, we see that
	\begin{align*}
	&\sum_{0\leq n\leq S_0-1}|\<F,F_{0,n}\>|^2 +\sum_{p\in\P_N}\sum_{0\leq n\leq S_p-1} |\<F,F_{p,n}\>|^2 \\
	&\quad= \frac{1}{\vol(N)} \sum_{-\overline N\leq m\leq \overline N} |\hat F(m)|^2\ \( |\hat{F_0}(m)|^2+\sum_{p\in\P_N} |\hat{F_p}(m)|^2\) 
	= \|F\|_2^2.
	\end{align*}
\end{proof}

\section{Relationship with directional representations}
\label{section directional}

While our construction of rotational uniform covering frames is motivated by neural networks, since they are also time-frequency representations and are partially generated by rotations, it is natural to ask whether they are related to recent developments in directional Fourier analysis. In order to make precise comparisons, we first summarize several important works on directional Fourier and wavelet analysis. 

\begin{enumerate}
	\item 
	Cand\`{e}s introduced a wavelet system that decomposes a function according to its direction, scale, and location. For an appropriate non-zero function $\psi$, see \cite[Definition 1]{candes1999harmonic}, and $(a,b,u)\in \R^+\times\R\times\s^{d-1}$, a \emph{ridgelet} is
	\begin{equation}
	\label{eq15}
	\psi_{a,b,u}(x)=a^{-1/2}\psi(a^{-1}(u\cdot x-b)).
	\end{equation}
	To see why this function is directionally sensitive, observe that $\psi_{a,b,u}$ is constant on hyperplanes perpendicular to $u$ and oscillates in the direction of $u$. Cand\`{e}s also introduced a discrete ridgelet system, and proved that both the continuous and discrete families are Parseval frames for $L^2(\R^d)$, see \cite[Theorems 1 and 2]{candes1999harmonic} for precise statements. 
	
	\item
	In contrast to the above wavelet inspired approach, Grafakos and Sansing constructed directional time-frequency representations. For a non-zero $g\in\S(\R)$ and $(m,t,u)\in\R\times\R\times\s^{d-1}$, a \emph{weighted Gabor ridge function} is
	\[
	g_{m,t,u}(x)
	=D^{(n-1)/2}(g_{m,t})(u\cdot x),
	\quad\text{where}\quad
	g_{m,t}(s)=e^{2\pi im\cdot(s-t)}g(s-t),
	\]
	and $D^{(n-1)/2}$ is the Fourier multiplier with symbol $|\xi|^{(n-1)/2}$. Similar to ridgelets, a weighted Gabor ridge function is constant along hyperplanes perpendicular to $u$ and oscillates in the direction of $u$. They proved that this continuous family is a Parseval frame for $L^2(\R^d)$, see \cite[Theorem 3]{grafakos2008gabor}, but were unable to obtain a discrete Parseval frame from this continuous family. This discretization problem was partially resolved: by omitting the multiplier $D^{(n-1)/2}$, it is possible to obtain a discrete frame, consisting of (un-weighted) Gabor ridge functions, for certain subspaces of $L^2(\R^d)$, see \cite[Theorem 5.6]{czaja2016discrete} for a precise statement. 
	
	\item
	While both of the previous representations relied on the ``ridge" function $x\mapsto x\cdot u$, Cand\`{e}s and Donoho used an entirely different approach to construct a family of wavelets called \emph{curvelets} \cite{candes2004new}. They are constructed using a decomposition of the Fourier domain in the same spirit as the \emph{second dyadic decomposition} \cite[pages 377 and  403]{stein1993harmonic}. Each curvelet oscillates in a certain direction and its shape satisfies the anisotropic scaling relation, width $\sim$ length$^2$. We omit their precise definitions since they are quite technical to state and such details are not relevant to our current discussion. 
	
	\item
	Finally, shearlets \cite{labate2005sparse, guo2007optimally,czaja2012isotropic, czaja2014anisotropic, bosch2015anisotropic} are also wavelets that extract directional information from functions. Unlike curvelets, they are generated using shearing operations as opposed to rotations. 
\end{enumerate}

With these examples of directional representations in mind, we return our attention to rotational uniform covering frames. We first show that each frame element oscillates in a certain direction, which is not surprising since they are partially generated by rotations. We already observed that $f_p\in\S(\R^d)$, so by Fourier inversion and a change of variables, for all $x\in\R^d$,
\begin{equation}
	\label{eq9}
	f_{m,r}(x)
	=\int_0^\infty \int_{\s^{d-1}} e^{2\pi i\rho(x\cdot r\omega)}\eta_A(\rho-Am)\beta_{m,B}(\omega)\rho^{d-1}\ d\sigma(\omega) d\rho,
\end{equation}
where $\sigma$ is the surface measure of $\s^{d-1}$. Recall that $\beta_{m,B}$ is supported in the cap $\{\phi_1\colon |\phi_1|\leq B_m\}$, which implies the integral in (\ref{eq9}) is taken over a subset of the sphere where $\omega\sim e_1$. We consider two separate cases. 

\begin{enumerate}[(a)]
	\item 
	If $x\in\R^d$ is parallel to $re_1$, then $x\cdot r\omega\sim 1$ and the phase in the integrand of (\ref{eq9}) rapidly changes. Since $\eta_A$ and $\beta_{m,B}$ are non-negative, we expect $f_{m,r}$ to oscillate in the direction of $re_1$. Further, due to the compact support of $\eta_A$, we have $\rho\sim mA$, so $f_{m,r}$ oscillates at frequency approximately $mA$.
	\item
	If $x\in\R^d$ is perpendicular to $re_1$, then $x\cdot r\omega\sim 0$, and the phase in the integrand of (\ref{eq9}) changes slowly. Since $\eta_A$ and $\beta_{m,B}$ are essentially constant, we do not expect $f_{m,r}$ to oscillate in directions perpendicular to $re_1$. 
\end{enumerate}

Figure \ref{fig 1} contains several numerically computed examples of $f_{m,r}$ and they agree with the qualitative description that we presented. 

\begin{figure}[!]
	\begin{center}
	\includegraphics[scale=0.5]{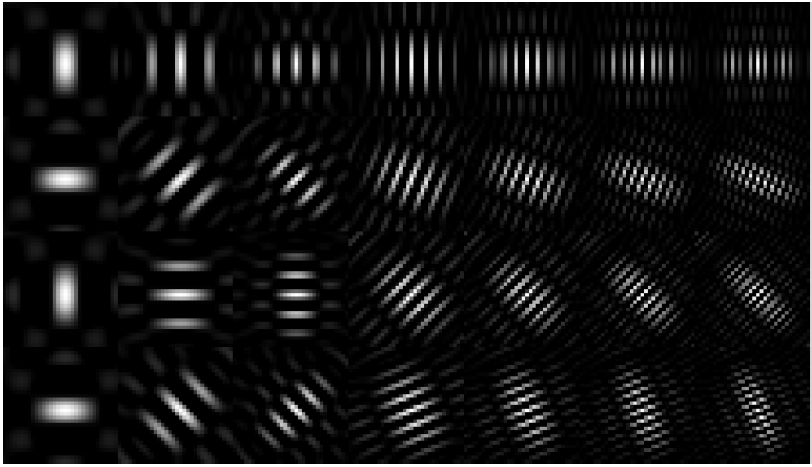}	
	\end{center}
	\caption{Let $d=2$ and $G$ be the group generated by rotations by $2\pi/4$. The figures are zoomed-in intensity plots of the rotational uniform covering frame elements. The $(n,m)$-th image corresponds to $f_{m,r_n}$ where $r_n=2\pi (n-1)/(4m^*)$.}
	\label{fig 1}
\end{figure}

In view of the directional bias of the frame elements, one might suspect that the frame coefficient $f*f_{m,r}$ carries directional frequency information about $f$. Indeed, by Plancherel,
\[
\|f*f_{m,r}\|_{L^2}^2
=\int_0^\infty \int_{\s^{d-1}} |\hat f(\rho r\omega)|^2 |\eta_A(\rho-mA)|^2 |\beta_{m,B}(\omega)|^2\rho^{d-1} \ d\sigma(\omega) d\rho. 
\]
By the definitions of $\eta_A$ and $\beta_{m,B}$, the integral is taken over the region where $\rho\sim mA$ and $\omega\sim e_1$. This equation shows that $\|f*f_{m,r}\|_{L^2}$ is a weighted average of how much $f$ oscillates at frequencies roughly $mA$ and in approximately the $re_1$ direction.  

We have shown that rotational uniform covering frame elements are directionally sensitive and their frame coefficients carry directional information. These properties are also shared by ridgelets, weighted Gabor ridge functions, and curvelets. However, since ridgelets and weighted Gabor ridge functions are constant on hyperplanes, they are not in $L^p(\R^d)$ for any $0<p<\infty$. In contrast, curvelets and rotational uniform covering frame elements are smooth and well-localized in both space and frequency, and consequently, they belong to many popular function spaces. Of course, curvelets are not uniform covering frames, but their constructions have significant similarities. This can be immediately seen by comparing Figure \ref{fig 1} with \cite[Figure 2.2]{candes2004new}. In view of these observations, it is reasonable to say that Gabor ridge functions are time-frequency analogues of ridgelets, while rotational uniform covering frame elements are time-frequency analogues of curvelets.

\section{Acknowledgements}

This work was supported in part by the Defense Threat Reduction Agency Grant HDTRA 1-13-1-0015 and by the Army Research Office Grants W911 NF-15-1-0112 and W911 NF-16-1-0008. Weilin Li was supported by the National Science Foundation Grant DMS-1440140 while in residence at the Mathematical Sciences Research Institute in Berkeley, California, during the Spring 2017 semester. 


\bibliography{rotUCFbib}
\bibliographystyle{plain}

\end{document}